\documentclass[a4paper,11pt]{article}
\usepackage[left=3.17cm,right=3.17cm,top=2.54cm,
headheight=0.5cm,headsep=0.54cm,bottom=2.54cm,footskip=0.79cm
]{geometry}

\usepackage{amsmath,amssymb,latexsym,amsthm,amscd,stmaryrd,amsfonts,mathrsfs}
\usepackage[title,titletoc,toc]{appendix}
\usepackage[all]{xy}
\usepackage{soul}

\theoremstyle{definition}
\newtheorem{Def}[subsubsection]{Definition}
\newtheorem{example}[subsubsection]{Example}
\newtheorem{rem}[subsubsection]{Remark}
\theoremstyle{plain}
\newtheorem{prop}[subsubsection]{Proposition}
\newtheorem{thm}[subsubsection]{Theorem}
\newtheorem{lem}[subsubsection]{Lemma}
\newtheorem{cor}[subsubsection]{Corollary}

\def\e{\varepsilon}

\usepackage{color}
\newcommand{\nc}{\newcommand}
\nc{\redtext}[1]{\textcolor{red}{#1}}
\nc{\bluetext}[1]{\textcolor{blue}{#1}}
\nc{\greentext}[1]{\textcolor{green}{#1}}
\nc{\yl}[1]{\redtext{From yq: #1}}
\nc{\zb}[1]{\redtext{From zb: #1}}
\begin{document}

\title{ A Glimpse To Quantum Cluster Superalgebras}

\author{Haitao Ma$^{1}$ $\thanks{e-mail: mahaitao871219@163.com}$, Yanmin Yang$^{2}$ $\thanks{e-mail: 491895274@qq.com}$, Zhu-Jun Zheng$^{1}$ $\thanks{e-mail: zhengzj@scut.edu.cn}$
\\[.3cm]
{\small $^{1}$School of Mathematics, South China University of Technology,
Guangzhou 510641, China}\\[.3cm]
{\small $^2$Department of Mathematics, Guangzhou University, 510006, China}
}
\maketitle

\begin{abstract}
In this paper, we introduce and study the quantum deformations of the cluster superalgebras. Then we prove the quantum version of the Laurent phenomenon for the super-case.
\end{abstract}

{Key words:} Cluster Superalgebras, Laurent Phenomenon

\section{Introduction}

Cluster algebras were discovered by S.Fomin and A.Zelevinsky~\cite{FZ1}-\cite{FZ4}.
They are a family of commutative rings designed to serve as an algebraic framework for the
theory of total positivity and canonical bases in semisimple groups and their quantum analogs.
The quantum deformation of cluster algebras were introduced by A.Berenstein and A.Zelevinsky\cite{BZ1}.
So far, lots of mathematicians discovered a profound connection between quantum groups and (quantum)cluster algebras.
For example, D.Hernandez and B.Leclerc used the cluster algebras to study the Grothendieck rings  of the certain monoidal subcategories of the representation of quantum affine algebras \cite{HL1}. And they studied the t-deformation $\mathcal{K}_t$ of the Grothendieck ring of  the subcategory mentioned above. They obtained that there was a quantum cluster algebra structure on $\mathcal{K}_t$\cite{HL2}. D.Hernandez and B.leclerc also used the cluster algebra algorithm to calculate q-characters of Kirillov-Reshetikhin modules for any untwisted quantum affine algebra $U_q(\hat{g})$ \cite{HL3}.
H.Nakajiama embed cluster algebras into the Grothendieck rings $\mathcal{R}$ of the categories of representations
of quantum loop algebras $U_q(L_g)$ of a symmetric Kac-Moody Lie algebra by the way of perverse sheaves on graded quiver varieties\cite{N1}.
H.Nakajima proposed an approach to Geiss-Leclerc-Schroer's conjecture on the cluster algebra structure on the coordinate ring of an unipotent subgroup and the dual canonical base\cite{N2}. The quantum version of the work in \cite{N1} was obtained by Y.Kimura and F.Qin \cite{KQ1}.

 Recently, V.Ovsienko introduced the cluster superalgebras when he studied the supercase of Coxeter's frieze patterns\cite{MOT1}\cite{O1}. The natural questions are what the quantum deformation of the cluster superalgebras are and if there are  some connections between the quantum cluster superalgebras and the quantum  affine superalgebras similar to the classical case. In order to solve the questions mentioned above, the first step we need to do is giving a reasonable concept of quantum cluster superalgeras.

Our approach to quantum cluster superalgebras is similar to the approach appear in \cite{BZ1}.  The cluster superalgebra structure is
completely determined by an extend quiver that encodes all the exchange relations. Now the quantum deformation of cluster superalgebras $\mathcal{A}$ is a $Q(q)$-algebra obtained by making each cluster into a super-quasi-commuting family $\{X_1,X_2,\cdots,X_{n+m}\}$; this means that $X_iX_j = (-1)^{\tau(e_i,e_j)}q^{\lambda_{ij}}X_jX_i$ for a
skew-symmetric integer $(m+n)\times (m+n)$ matrix $\Lambda = (\lambda_{ij})$ . In doing so, we have to modify the mutation process and the exchange relations so that all the adjacent quantum
super-clusters will also be quasi-super-commuting.
This imposes the compatibility relation between the quasi-commutation matrix $\Lambda$ and the exchange matrix $B$. Then we have to develop a formalism that allows us to show that any compatible matrix pair $(\widetilde{\mathcal{Q}}, \Lambda)$ give rise to a well defined quantum cluster superalgebra.

The paper is organized as follows. In section 2 we present necessary definitions and facts from the theory of cluster superalgebras. In section 3 and section 4, we introduce compatible matrix pairs $(\widetilde{\mathcal{Q}}, \Lambda)$ and their mutations, and how the compatible matrix pair $(\widetilde{\mathcal{Q}}, \Lambda)$ gives a well defined quantum cluster superalgebra. Section 5  introduces the Laurent phynomenon of quantum cluster superalgebras.

\section{Preliminary}

In this section, we introduce the notion of cluster superalgebra.
Everywhere in this paper, the odd coordinates are frozen. All the concepts and results in this section are in the paper\cite{O1}.

\subsection{The Extended Quiver}

We introduce the notion of extended quiver following V.Ovsienko's work \cite{O1}.

\begin{Def}
Given a quiver $\mathcal{Q}$  with no loops and no $2$-cycles,
an  extended quiver $\widetilde{\mathcal{Q}}$ with  underlying quiver $\mathcal{Q}$,
is a quiver defined as follows.

{\bf A}.
$\widetilde{\mathcal{Q}}$ has $m$ extra ``colored'' vertices labeled by the odd coordinates
$\{\xi_1,\ldots,\xi_m\}$, so that
$$
\widetilde{\mathcal{Q}}_0=\{x_1,\ldots,x_n,\xi_1,\ldots,\xi_m\}.
$$

{\bf B}.
Some of the new vertices $\{\xi_1,\ldots,\xi_m\}$ are related to
the vertices $\{x_1,\ldots,x_n\}$ of the underlying quiver $\mathcal{Q}$ by ingoing or outgoing arrows. That is,
$$
\widetilde{\mathcal{Q}}_1=\mathcal{Q}_1\cup_{k}\{\xi_i\to{}x_k,i\in{}I_k\,,x_k\to\xi_j,j\in{}J_k\}.
$$
For every $1\leq{}k\leq{}n$, $I_k\cap{}J_k=\emptyset$,
where $I_k = \{ i | \xi \to{} x_k \in \widetilde{\mathcal{Q}}_1\}$ and $J_k = \{ j | x_k \to{} \xi_j \in \widetilde{\mathcal{Q}}_1\}$.

\end{Def}

Let $\widetilde{\mathcal{Q}}_2 = \{\xi_i\to{}v_k\to{}\xi_j | \xi_1, \xi_3 \in \mathcal{Q}_1, v_k \in \mathcal{Q}_0\}$,where
$\mathcal{Q}_0 = \{x_1, x_2, \cdots, x_n\}$, $\mathcal{Q}_1 = \{\xi_1, \xi_2, \cdots, \xi_m\}$.

\begin{Def}
Given an extended quiver~$\widetilde{\mathcal{Q}}$ and an even vertex $x_k\in\mathcal{Q}_0$,
the mutation~$\mu_k$ is defined by the
following rules:
\begin{enumerate}
\item[(0)]
the underlying quiver $\mathcal{Q}\subset\widetilde{\mathcal{Q}}$ mutates
according to the same rules as the classical case;
\item[(1)]
given a $2$-path $(\xi_i\to x_k\to\xi_j)$,
add the $2$-paths $(\xi_i\to x_\ell\to\xi_j)$ for all $x_\ell\in\mathcal{Q}_0$
connected to~$x_k$ by an arrow $(x_k\to{}x_\ell)$, we define:
$$
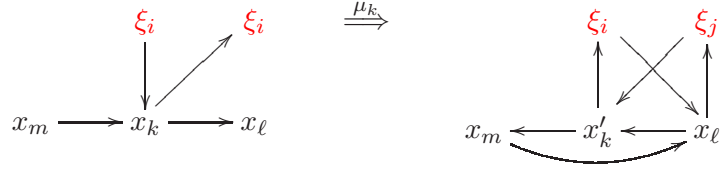
\xymatrix{
&{\color{red}\xi_i}\ar@{->}[d]&
{\color{red}\xi_i}\ar@<-2pt>@{<-}[ld]\\
x_m&x_k\ar@{<-}[l]&x_\ell\ar@{<-}[l]
}
\qquad
\stackrel{\mu_k}{\Longrightarrow}
\qquad
 \xymatrix{
&{\color{red}\xi_i}\ar@<2pt>@{->}[rd]\ar@{<-}[d]&
{\color{red}\xi_j}\ar@<-2pt>@{->}[ld]\ar@{<-}[d]\\
x_m\ar@/^-1pc/[rr]&x_k'\ar@{->}[l]&x_\ell\ar@{->}[l]
}
$$
\item[(2)]  reverse all the arrows at $x_k$;
\item[(3)]  remove all the odd-even-odd $2$-paths (if any)
with opposite orientations:
$(\xi_i\to{}x_\ell\to\xi_j)$ and
$(\xi_i\leftarrow{}x_\ell\leftarrow\xi_j)$ created by rule (1).
\end{enumerate}
\end{Def}

Observe that every mutation $\mu_k$ of the quiver~$\widetilde{\mathcal{Q}}$ is an involution.

\begin{Def}
A mutation of the extended quiver $\widetilde{\mathcal{Q}}$
at a vertex $x_k\in\mathcal{Q}_0$ is  allowed if
every vertex $\xi_i\in{}I_\ell\cup{}I_k$ is connected to
every vertex $\xi_j\in{}J_\ell\cup{}J_k$ by a $2$-path
$(\xi_i\to{}x_\ell\to\xi_j)$ in the resulting quiver~$\mu_k(\widetilde{\mathcal{Q}})$.
\end{Def}

The following statement provides with a necessary and sufficient condition
for the mutation at a vertex $x_k$ to be allowed.

\begin{lem}
Given an extended quiver $\widetilde{\mathcal{Q}}$,
the mutation at a given vertex $x_k$ is allowed if and only if
for every vertex $x_\ell\in\mathcal{Q}$ connected to $x_k$ by an outgoing arrow $x_k\to{}x_\ell$,
(at least) one of the following conditions is satisfied:

\begin{enumerate}

\item[(a)]
$I_k=I_\ell$;

\item[(b)]
$J_k=J_\ell$;

\item[(c)]
$I_k=J_k=\emptyset$.

\item[(d)]
$I_k=J_\ell$, and $J_k=I_\ell$;

\item[(e)]
$I_\ell=J_\ell=\emptyset$.
\end{enumerate}

\end{lem}

\subsection{The Algebra $A(\widetilde{\mathcal{Q}})$}

Everywhere in this section, we assume that the mutations are allowed.

\begin{Def}
Given an extended quiver $\widetilde{\mathcal{Q}}$,
the mutation $x'_k=\mu_k(x_k)$ of the indeterminate~$x_k$
 is defined by the following formula
\begin{equation}
\label{Mute}
x_k':=
\frac{1}{x_k}
\left(
\prod\limits_{\substack{x_i\to x_k }}x_i
+\prod\limits_{\substack{x_j\leftarrow x_k }}x_j+
\Big(
\sum\limits_{\substack{\xi_i\to{}x_k}}\xi_i
\Big)\Big(
\sum\limits_{\substack{\xi_j\leftarrow{}x_k}}\xi_j\Big)
\prod\limits_{\substack{x_i\to x_k }}x_i
\right),
\end{equation}
that will be called, as in the classical case, an exchange relation.
\end{Def}

Note that,
since the odd coordinates anticommute,
the same holds for the sums:
$$
(\sum\xi_i)(\sum\xi_j)=-(\sum\xi_j)(\sum\xi_i).
$$

\begin{enumerate}
\item
The superalgebra, $A(\widetilde{\mathcal{Q}})$
generated by
the coordinates $\{x_1,\ldots,x_n,\xi_1,\ldots,\xi_m\}$, as well as
all possible mutations of the even coordinates: $x'_1,\ldots,x'_n,x''_1,\ldots$ will be called
the cluster superalgebra associated to the quiver~$\widetilde{\mathcal{Q}}$.
It will be usually considered over $\mathbb{C}$, yet other choices of the ground fields are possible.
\item
The pair $(\widetilde{\mathcal{Q}},\{x_1,\ldots,x_n,\xi_1,\ldots,\xi_m\})$ is called
the initial seed of the cluster superalgebra.
A mutation gives rise to a new seed.

\end{enumerate}

Note that, unlike the classical case, the above mutation of $x_k$ is not an involution.

\section{Compatible Pairs}

Let $\widetilde{\mathcal{Q}}$ be an extended quiver with $n$ even vertices $\mathcal{Q}_0 = \{x_1, x_2, \cdots, x_n\}$ and $m$ odd vertices $\mathcal{Q}_1 = \{x_{n+1}, x_{n+2}, \cdots, x_{n+m}\}$. Assume $l \leq n$ and $\{x_{l+1}, x_{l+2}, \cdots, x_{n+m}\}$ is frozen,  Let $B$ be the $ n \times l$ matrices defined by
$$b_{ij} = \sharp\{x_i \rightarrow x_j\} - \sharp\{x_j \rightarrow x_i\}.$$

\begin{Def}
Let $\widetilde{\mathcal{Q}}$ be an extended quiver with $n$ even vertices and $m$ odd vertices. Let $\Lambda$ be a skew-symmetric $(m + n) \times (m + n)$ integer matrix with rows and columns labeled by $[1 , m + n]$,
$
\Lambda
=\left(
   \begin{array}{cc}
     \Lambda_{11} & \Lambda_{12} \\
     \lambda_{21} & \lambda_{22} \\
   \end{array}
 \right)
$
where $\Lambda_{11}, \Lambda_{12}, \Lambda_{21}, \Lambda_{22}$ are $n \times n, n \times m, m \times n, m \times m$ matrices. We say a pair $(\widetilde{\mathcal{Q}}, \Lambda)$ is compatible pair if it satisfied the following conditions.

{\bf A}.
$D = B^{T}\left(
            \begin{array}{cc}
              \Lambda_{11} & \Lambda_{12} \\
            \end{array}
          \right)
$, $D$ is diagonal matrix with positive integers diagonal entries $d_j$.

{\bf B}.
 If $x_k,x_l \in \mathcal{Q}_1$, and $x_k$ and $x_l$ are connected though a 2-path in $\widetilde{\mathcal{Q}}_2$, then for any $i \neq k, l$, we have $\lambda_{i,k} = - \lambda_{i,l}$.
\end{Def}

We will extend mutations to these of compatible pairs. Fix a sign $\varepsilon \in \{+ 1, - 1\}$, if a mutation of the extended quiver $\widetilde{\mathcal{Q}}$ at a vertex $x_k \in \mathcal{Q}_0$ is allowed, we set
$$\widetilde{\mathcal{Q}}' = \mu_k(\widetilde{\mathcal{Q}}).$$

$E_{\varepsilon}$ is the $n \times n$ matrix with entries

$$e_{ij}
=\left\{\begin{array}{ll}
\delta_{ij}
& \text{if} \ j \neq k; \\[.05in]
-1
& \text{if} \ i = j = k; \\[.05in]
max\{0, - \varepsilon b_{ik}\}
& \text{if} \ i \neq j = k.
 \end{array}   \right.
$$

Now suppose $(\Lambda, \widetilde{\mathcal{Q}})$ is compatible. We set

$$\Lambda' = \left(
               \begin{array}{cc}
                 E_{\varepsilon}^T & 0 \\
                 0 & id \\
               \end{array}
             \right)
\Lambda
 \left(
               \begin{array}{cc}
                 E_{\varepsilon} & 0 \\
                 0 & id \\
               \end{array}
             \right)
$$

Thus, $\Lambda'$ is still skew-symmetric.
\begin{prop}
(1) The pair $(\Lambda',\widetilde{\mathcal{Q}}')$ is compatible.

(2)$\Lambda'$ is independent of choice of a sign $\varepsilon$.
\end{prop}

\begin{proof}
 Let $F_{\varepsilon}$ is the $l \times l$ matrix with entries

$$f_{ij}
=\left\{\begin{array}{ll}
\delta_{ij}
& \text{if} \ i \neq k; \\[.05in]
-1
& \text{if} \ i = j = k; \\[.05in]
max\{0, \varepsilon b_{kj}\}
& \text{if} \ i = k \neq j.
 \end{array}   \right.
$$
By some direct computation, we have

$$E_{\varepsilon}^2 = 1,\ F_{\varepsilon}^2 = 1,\ F_{\varepsilon}^{T}D_{11} = D_{11} E_{\varepsilon},$$
where $D=\left(
           \begin{array}{cc}
             D_{11} & D_{12} \\
           \end{array}
         \right)
$, $D_{11}$ is the $n \times n$ matrix.

Then we need to show $(\Lambda',\widetilde{\mathcal{Q}}')$ satisfies the two conditions of compatible pairs.

\begin{eqnarray*}
 &&(B')^{T}(\Lambda_{11}' \ \Lambda_{12}') \\
 & = & (E_{\varepsilon}BF_{\varepsilon})^{T}(E_{\varepsilon}^{T}\Lambda_{11}E_{\varepsilon} \ E_{\varepsilon}^{T}\Lambda_{12}) \\
& = & F_{\varepsilon}B^{T}(\Lambda_{11}E_{\varepsilon} \  \Lambda_{12})\\
& = &  (F_{\varepsilon}^{T}D_{11}E_{\varepsilon} \  0)\\
& = & (D_{11} \ 0).
\end{eqnarray*}

That is, $(\Lambda',\widetilde{\mathcal{Q}}')$ satisfies the first condition.
Since $\Lambda_{12}' = E_{\varepsilon}^{T}\Lambda_{12}$. If $1 \leq i \leq n, n \leq j \leq n+m,$

$$\lambda_{ij}'
=\left\{\begin{array}{ll}
\lambda_{ij}
& \text{if} \ i \neq k; \\[.05in]
\sum\limits_{i \in \{1 \leq i \leq n | i \neq k, \varepsilon b_{ik} > 0\}}\varepsilon b_{ik}\lambda_{ik}
& \text{if} \ i = k.
 \end{array}   \right.
$$

$\Lambda$ satisfies the condition of compatible pair. Therefore, $\Lambda'$ automatically satisfies the condition. The first part of the proposition follows.
And the second part of the proposition follows directly from the formula of $\Lambda_{12}'$.
\end{proof}

\begin{Def}
Let $(\Lambda,\widetilde{\mathcal{Q}})$ be a compatible pair. If a mutation of $\widetilde{\mathcal{Q}}$ at a vertex $x_k \in \widetilde{\mathcal{Q}}_0$ is allowed, we say $(\Lambda',\widetilde{\mathcal{Q}}')$ is obtained from $(\Lambda,\widetilde{\mathcal{Q}})$ by the mutation at vertex $x_k$, and write $(\Lambda',\widetilde{\mathcal{Q}}')= \mu_k(\Lambda,\widetilde{\mathcal{Q}}).$
\end{Def}

\begin{prop}
The mutation of $\Lambda$ is involutive: for any compatible pair $(\Lambda, \widetilde{\mathcal{Q}})$, we have $\mu_k(\mu_k(\Lambda)) = \Lambda$.
\end{prop}
\begin{proof}
Let $(\Lambda', \widetilde{\mathcal{Q}}') = \mu_k((\Lambda, \widetilde{\mathcal{Q}}))$, we have $B' = E_{\varepsilon}BF_{\varepsilon}$. The k-th column of $B'$ is the negative of the k-th column of $B$. It follows that $E_{\varepsilon}' = - E_{\varepsilon}$.
\begin{eqnarray*}
 &&\mu_k(\mu_k(\Lambda)) \\
 & = & \left(
         \begin{array}{cc}
           E_{\varepsilon}' & 0 \\
           0 & id \\
         \end{array}
       \right)
 \left(
   \begin{array}{cc}
     E_{\varepsilon} & 0 \\
     0 & id \\
   \end{array}
 \right)
 \Lambda
  \left(
   \begin{array}{cc}
     E_{\varepsilon} & 0 \\
     0 & id \\
   \end{array}
 \right)
 \left(
         \begin{array}{cc}
           E_{\varepsilon}' & 0 \\
           0 & id \\
         \end{array}
       \right)
  \\
& = & \Lambda.
\end{eqnarray*}
\end{proof}

\section{Quantum Cluster Superalgebra Setup}
\subsection{Based Quantum Supertorus}

Let $L$ be a $\mathbb{Z}_2$-grading lattice of rank $n + m$, with a skew-symmetric bilinear form $\Lambda: L \times L \rightarrow \mathbb{Z}$. Let $\{e_1, e_2, \cdots,e_{n+m}\}$ be a  basis of L .
$$deg(e_{i})
=\left\{\begin{array}{ll}
0
& \text{if} \ 1\leq i \leq n; \\[.05in]
1
& \text{if}  \ i \geq n.
 \end{array}   \right.
$$
Let $q$ be a formal variable, and $\mathbb{Z}[q^{\pm\frac{1}{2}}] \subset \mathbb{Q}(q^{\frac{1}{2}})$ denote the ring of integer Laurent polynomials in the variable $q^{\frac{1}{2}}$.

\begin{Def}
The based quantum supertorus associated with $L$ is $\mathbb{Z}[q^{\frac{1}{2}}]$-algebra $\mathcal{T} = \mathcal{T}(\Lambda)$ with a distinguished $\mathbb{Z}[q^{\frac{1}{2}}]$-basis $\mathfrak{s} = \{X^e | e \in L $ such that $ a_i \in \{0,1\} \ \forall \ i > n\}$ and multiplication

$$X^eX^f = (- 1)^{\tau(e,f)}q^{\frac{\Lambda(e,f)}{2}}X^{e + f}, (e,f \in L),$$
where $X^{e + f} = 0$ if $(e + f)_j \geq 2$ for some $j > n$.

$$\tau(e,f) = \sharp \{(j_1,j_2) | j_1 > j_2 , e^{(1)}_{j_1} \neq 0, f^{(1)}_{j_2} \neq 0 \}.$$
\end{Def}

\begin{prop}
(1). $\mathcal{T}$ is associative algebra: we have

$$(X^eX^f)X^g = X^e(X^fX^g).$$

(2). The basis elements satisfy the commutation relations:

$$X^eX^f = (-1)^{\tau(e,f) + \tau(f,e)}q^{\Lambda(e,f)}X^fX^e.$$

(3). $\mathcal{T}$ is superalgebra generated by the even elements $\{X^{e_i} | 1 \leq i \leq n \}$ and the odd elements $\{X^{e_i} | i > n\}$ satisfied the following relations.

\begin{eqnarray*}
  (R1)&& X^{e_i}X^{e_j} = q^{\lambda_{ij}}X^{e_j}X^{e_i} , ~\rm{if} \  i \leq n \ \rm{or} \ j \leq n ;\\
  (R2)&& X^{e_i}X^{e_j} = -q^{\lambda_{ij}}X^{e_j}X^{e_i}, ~\rm{if} \  i > n \ \rm{and} \ j > n;\\
  (R3)&& (X^{e_i})^2 = 0, ~\rm{if} \  i > n.
\end{eqnarray*}

\end{prop}

\begin{proof}
Since
\begin{eqnarray*}
 (X^eX^f)X^g & = & (-1)^{\tau(e,f)}q^{\frac{\Lambda(e,f)}{2}}X^{e + f}X^g\\
& = & (-1)^{\tau(e,f) + \tau(e+f,g)}q^{\frac{\Lambda(e,f)}{2}}q^{\frac{\Lambda(e + f,g)}{2}}X^{e+f+g}\\
& = &  (-1)^{\tau(e,f) + \tau(e+f,g)}q^{\frac{\Lambda(e,f) + \Lambda(f,g) + \Lambda(e,g)}{2}}X^{e+f+g};
\end{eqnarray*}

\begin{eqnarray*}
 X^e(X^fX^g) & = & (-1)^{\tau(f,g)}q^{\frac{\Lambda(f,g)}{2}}X^{e}X^{f + g}\\
& = & (-1)^{\tau(f,g) + \tau(e,f+g)}q^{\frac{\Lambda(f,g)}{2}}q^{\frac{\Lambda(e,f+g)}{2}}X^{e+f+g}\\
& = &  (-1)^{\tau(f,g) + \tau(e,f + g)}q^{\frac{\Lambda(e,f) + \Lambda(f,g) + \Lambda(e,g)}{2}}X^{e+f+g};
\end{eqnarray*}

It is obvious that $\tau(e,f) + \tau(e+f,g) = \tau(f,g) + \tau(e,f + g)$. Thus, we have proved the first part of the proposition.

Secondly,
$$X^eX^f = (-1)^{\tau(e,f)}q^{\frac{\Lambda(e,f)}{2}}X^{e + f};$$
$$X^fX^e = (-1)^{\tau(f,e)}q^{\frac{\Lambda(f,e)}{2}}X^{e + f}.$$

Therefore, $X^eX^f = (-1)^{\tau(e,f) + \tau(f,e)}X^fX^e.$

At last, we will prove the last part of the proposition. If $e = \sum\limits_{i = 1}^{n + m}a_ie_i$,

$$X^e = q^{\frac{1}{2}\sum\limits_{l < k}a_ka_l\lambda_{kl}}(x^{e_1})^{a_1}\cdots (x^{e_{n + m}})^{a_{n}}.$$

Thus, $\mathcal{T}$ is generated by the element $\{X^{e_i} | 1 \leq i \leq n + m \}$. By some direct computations, the generated relation is obtained.
\end{proof}

Set $\mathfrak{s}_1 = \{X^e | e \in L $ such that $ \forall  \ i > n, a_i = 0 \} \subset \mathfrak{s}$, and  $S = \{t \in \mathcal{T} | t = t_1 + t_2, t_1 \in \mathcal{T}_1,  t_2 \in \mathcal{T}_2, t_1 \neq 0\}$, where $\mathcal{T}_1 $ is a vector space with a distinguished $\mathbb{Z}[q^{\frac{1}{2}}]$-basis $\mathfrak{s}_1$, $\mathcal{T}_2 $ is a vector space with a distinguished $\mathbb{Z}[q^{\frac{1}{2}}]$-basis $\mathfrak{s} \setminus \mathfrak{s}_1$. Set $\mathcal{F} = S^{-1}\mathcal{T}$. $\mathcal{F}$ is a $\mathbb{Q}[q^{\frac{1}{2}}]$-algebra. A quantum cluster superalgebra we defined below is the $\mathbb{Z}[q^{\frac{1}{2}}]$-subalgebra of  $\mathcal{F}$.

\begin{Def}
A supertoric frame in $\mathcal{F}$ is a mapping $M: \mathbb{Z}^{n|m} \rightarrow \mathcal{F} - \{0\}, ~c \mapsto \varphi(X^{\eta(c)})$, where $\varphi$ is automorphism of superalgebra $\mathcal{F}$, and $\eta: \mathbb{Z}^{n|m} \rightarrow L$ is an isomorphism of super-lattice.
\end{Def}

Let $\Lambda_M$ be the bilinear form on  $\mathbb{Z}^{n|m}$ obtained by transferring the form $\Lambda$ from $L$ by $\eta$, and $\{e_1, e_2, \cdots,e_{n+m}\}$ be a standard basis of   $\mathbb{Z}^{n|m}$. The multiplication is given by

$$M(c)M(d) = (-1)^{\tau(c,d)}q^{\frac{\lambda_M(c,d)}{2}}M(c+d).$$

Similarly as above, we know the supertoric frame $M$ is uniquely determined by $X_i = M(e_i)$ for $i \in [1, m+n]$.

 \subsection{Quantum Super-Seed}

\begin{Def}
A quantum super-seed is a pair$(M,\widetilde{\mathcal{Q}})$, where

(1) M is a super-toric frame in $\mathcal{F}$.

(2)$\widetilde{\mathcal{Q}}$ is an extended quiver with $n$ even vertices and $m$ odd vertices.

(3)$(\Lambda_M,\widetilde{\mathcal{Q}})$ is compatible.

\end{Def}

Let $(M,\widetilde{\mathcal{Q}})$ be a quantum super-seed. If the mutation of $\widetilde{\mathcal{Q}}$ at vertex $x_k$ is allowed, we define $M' : \mathbb{Z}^{n|m} \rightarrow  \mathcal{F} - \{0\}$  by

$$M'(e_i)
=\left\{\begin{array}{ll}
M(e_i)
& \text{if} \ i \neq k; \\[.05in]
M(- e_k + \sum\limits_{b_{ik} > 0}b_{ik}e_i) + M(- e_k - \sum\limits_{b_{ik} < 0}b_{ik}e_i) \\+ \sum\limits_{ x_i \rightarrow x_k \rightarrow x_j \in \widetilde{\mathcal{Q}}_2}(-1)^{\tau (e_{i},e_{j})}M(- e_k + \sum\limits_{b_{ik} > 0}b_{ik}e_i + e_{i} + e_{j} )
& \text{if}  \ i = k,
 \end{array}   \right.
$$

and $\Lambda' = \left(
               \begin{array}{cc}
                 E_{\varepsilon} & 0 \\
                 0 & id \\
               \end{array}
             \right)
\Lambda
 \left(
               \begin{array}{cc}
                 E_{\varepsilon}^T & 0 \\
                 0 & id \\
               \end{array}
             \right)
: \mathbb{Z}^{n|m} \times \mathbb{Z}^{n|m} \rightarrow \mathbb{Z}$.

\begin{prop}
(1) The $M'$ is a super-toric frame.

(2)$(\Lambda_M',\mu_k(\widetilde{\mathcal{Q}}))$ is obtained by $(\Lambda_M,\widetilde{\mathcal{Q}})$ at the vertex $x_k$.

(3)The pair $(M',\mu_k(\widetilde{\mathcal{Q}}))$ is quantum super-seed.
\end{prop}

\begin{proof}
We need to compute the following commutation relation:
for any $i \neq k,$
\begin{eqnarray*}
&&M'(e_i)M'(e_k)\\
& = & M(e_i)[M(- e_k + \sum\limits_{b_{ik} > 0}b_{ik}e_i) + M(- e_k - \sum\limits_{b_{ik} < 0}b_{ik}e_i) \\
&&+ \sum\limits_{ x_i \rightarrow x_k \rightarrow x_j \in \widetilde{\mathcal{Q}}_2}(-1)^{\tau (e_{i},e_{j})}M(- e_k + \sum\limits_{b_{ik} > 0}b_{ik}e_i + e_{i} + e_{j} )]\\
& = &   (-1)^{\tau(e_i,e_k)}q^{\lambda'_{ik}}M'(e_k)M'(e_i).
\end{eqnarray*}

\end{proof}
By the above proposition, the following corollary can be obtained directly.
\begin{cor}
Let $(M,\widetilde{\mathcal{Q}})$ be a quantum super-seed. Suppose the quantum seed $(M',\mu_k(\widetilde{\mathcal{Q}}))$ is obtained from $(M,\widetilde{\mathcal{Q}})$ by the mutation in the direction k. Let $X_i = M(e_i), \ X_i' = M'(e_i)$, then $X_i' = X_i$ for $i \neq k$,

\begin{eqnarray*}
 X_k' & = & M(- e_k + \sum\limits_{b_{ik} > 0}b_{ik}e_i) + M(- e_k - \sum\limits_{b_{ik} < 0}b_{ik}e_i)\\
&& + \sum\limits_{\xi_i \rightarrow x_k \rightarrow \xi_j \in \widetilde{\mathcal{Q}}_2}(-1)^{\tau (e_{n+i},e_{n+j})}M(- e_k + \sum\limits_{b_{ik} > 0}b_{ik}e_i + e_{n+i} + e_{n+j} ).
\end{eqnarray*}
\end{cor}

\begin{rem}
Unlike the quantum cluster algebra, the mutation of quantum super-seed is not involutive. That is, $\mu_k(\mu_k(M)) \neq M$. But we have $\Lambda_{\mu_k(\mu_k(M))} = \Lambda_{M}$.
\end{rem}

\subsection{Quantum Cluster Superalgebra}
For a quantum initial seed $(M,\widetilde{\mathcal{Q}})$, we denote $\widetilde{\mathbf{X}} = \{X_1, X_2, \cdots ,X_{n + m}\}$, the corresponding element in $\mathcal{F}$ is given by $X_i = M(e_i)$. Set
$$\mathbf{C} = \{X_i | i \ is \ the \ frozen \ point \ of \ \widetilde{\mathcal{Q}} \ or \ the \ odd \ point\}.$$
We call the subset $\mathbf{X} = \widetilde{\mathbf{X}} - \mathbf{C} $ the cluster of quantum super-seed $(M,\widetilde{\mathcal{Q}})$.
By the definition of the quantum cluster superalgebra, the following proposition can be obtained obviously.
\begin{prop}
The set $\mathbf{C}$ is invariant under mutation.
\end{prop}

\begin{Def}
The quantum cluster superalgebra  $\mathcal{A}(\widetilde{\mathcal{Q}})$ is a $\mathbb{Z}(q^{\pm \frac{1}{2}})$ subalgebra of $\mathcal{F}$ generated by the union of clusters of all possible mutation sequence of $(M, \widetilde{\mathcal{Q}})$ together with all the element in $\mathbf{C}$.
\end{Def}

We have known the fact that the mutation of quantum super-seed is not involutive. But we have the following proposition.

\begin{prop}
\label{prop1}
If the mutation of $\widetilde{\mathcal{Q}}$ at the direction $k$ is allowed, then  $\mathcal{A}(\widetilde{\mathcal{Q}}) = \mathcal{A}(\mu_k^2(\widetilde{\mathcal{Q}}))$.
\end{prop}
\begin{proof}
One can easily obtains:

\begin{eqnarray*}
 X_k'' & = & X_k
+ \sum\limits_{\xi_i \rightarrow x_k \rightarrow \xi_j \in \widetilde{\mathcal{Q}}_2}(-1)^{\tau (e_{n+i},e_{n+j})}M( e_k  + e_{n+i} + e_{n+j} ).
\end{eqnarray*}
By some direct computation,
\begin{eqnarray*}
 X_k & = & X_k''(1
+ \sum\limits_{\xi_i \rightarrow x_k \rightarrow \xi_j \in \widetilde{\mathcal{Q}}_2}(-1)^{\tau (e_{n+i},e_{n+j})}M(e_{n+i} + e_{n+j} )).
\end{eqnarray*}
where $x''_k:=\mu_k^2(x_k)$.
This expression is a combination of the initial coordinates $X_k$. It belongs to the algebra generated by $X_k$.
\end{proof}
\subsection{Example}

\begin{example}
Our most elementary example is the quiver $\widetilde{\mathcal{Q}}$ with one even vertices $x_1$
and two odd vertices $x_2, x_3$.  Set 
 $M : \mathbb{Z}^{1|2} \rightarrow \mathcal{F}$, which is determined by $\Lambda_M = \left(
                                                                                       \begin{array}{ccc}
                                                                                         0 & \lambda_1 & -\lambda_1 \\
                                                                                         -\lambda_1 & 0 & \lambda_3 \\
                                                                                         \lambda_1 & -\lambda_3 & 0 \\
                                                                                       \end{array}
                                                                                     \right).
 $
 Though the directly computing, $(\Lambda_M, \widetilde{\mathcal{Q}} )$ is compatible. $(M, \widetilde{\mathcal{Q}})$ is the initial quantum super-seed.
$$
 
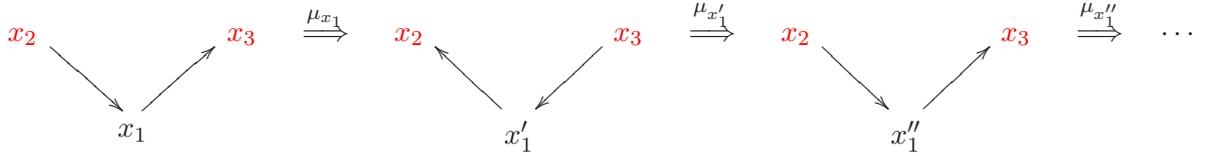
\xymatrix{
{\color{red}x_2}\ar@<3pt>@{->}[rd]&&
{\color{red}x_3}\ar@<-3pt>@{<-}[ld]\\
& x_1&
}
\quad
\stackrel{\mu_{x_1}}{\Longrightarrow}
\quad
 \xymatrix{
{\color{red} x_2}\ar@<3pt>@{<-}[rd]&&
{\color{red} x_3}\ar@<-3pt>@{->}[ld]\\
& x_1'&
}
\quad
\stackrel{\mu_{x_1'}}{\Longrightarrow}
\quad
 \xymatrix{
{\color{red} x_2}\ar@<3pt>@{->}[rd]&&
{\color{red} x_3}\ar@<-3pt>@{<-}[ld]\\
& x_1''&
}
\quad
\stackrel{\mu_{x_1''}}{\Longrightarrow}
\quad
\cdots
$$

Let $X_i = M(e_i)$ ,then we have $X_1X_2 = q^{\lambda_1}X_2X_1$, $X_1X_3 = q^{- \lambda_1}X_3X_1$, $X_2X_3 = (-1)q^{\lambda_3}X_3X_2$, $X_2^2 = X_3^2 = 0.$ By the directly computing, we have the following identity

\begin{eqnarray*}
 X1' & = & M(- e_1) + M(- e_1)+ (-1)^{\tau (e_{2},e_{3})}M(- e_1  + e_{2} + e_{3} )\\
 & = & X_1^{-1}(2 + q^{-\frac{\lambda_3}{2}}X_2X_3).
\end{eqnarray*}

Then, we have $\Lambda_{M'} = \left(
                                                                                       \begin{array}{ccc}
                                                                                         0 &  - \lambda_1 & \lambda_1 \\
                                                                                          \lambda_1 & 0 & \lambda_3 \\
                                                                                         - \lambda_1 & -\lambda_3 & 0 \\
                                                                                       \end{array}
                                                                                     \right)$
and $M'(e_1) =  X_1^{-1}(2 + q^{-\frac{\lambda_3}{2}}X_2X_3),M'(e_2) = X_2, M(e_3) = X_3$. We have

\begin{eqnarray*}
 X_1'' & = & M'(- e_1) + M'(- e_1)+ (-1)^{\tau (e_{3},e_{2})}M'(- e_1  + e_{2} + e_{3} )\\
 & = & (X_1')^{-1}(2 - q^{-\frac{\lambda_3}{2}}X_2X_3)\\
 & = & X_1(2 + q^{-\frac{\lambda_3}{2}}X_2X_3)^{- 1}(2 - q^{-\frac{\lambda_3}{2}}X_2X_3)\\
 & = & X_1(1 - q^{-\frac{\lambda_3}{2}}X_2X_3).
 \end{eqnarray*}

Similarly as above, we have the following identities.

$$X_1''' = X_1^{ - 1}(2 + 3q^{- \frac{\lambda_3}{2}}X_2X_3),$$
$$X_1'''' = X(1 - 2q^{- \frac{\lambda_3}{2}}X_2X_3).$$

Like the classical case, we can see that, the mutation is not involution, and the process is infinite and aperiodic. By the proposition $\ref{prop1}$, the corresponding quantum cluster superalgebra is a  $\mathbb{Z}[q^{\pm \frac{1}{2}}]$ algebra generated by $X_1^{\pm 1}, X_2, X_3$ with the following relations,

 $$X_1^{\pm 1}X_2 = q^{\pm \lambda_1}X_2X_1, \ X_1^{\pm}X_3 = q^{\pm\lambda_1}X_3X_1,\  X_2X_3 = (-1)q^{\lambda_3}X_3X_2, \ X_2^2 = X_3^2 = 0.$$
\end{example}

\begin{example}
Consider the quiver $\widetilde{\mathcal{Q}} $ with two even vertices $\{x_1,x_2\}$and two odd vertices$\{x_3,x_4\}$:
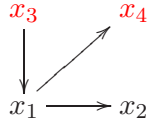
$$
\xymatrix{
{\color{red} x_3}\ar@{->}[d]&
{\color{red} x_4}\ar@<-1pt>@{<-}[ld]\\
x_1&x_2\ar@{<-}[l]
}
$$
Consider the following mutation sequence:

$$
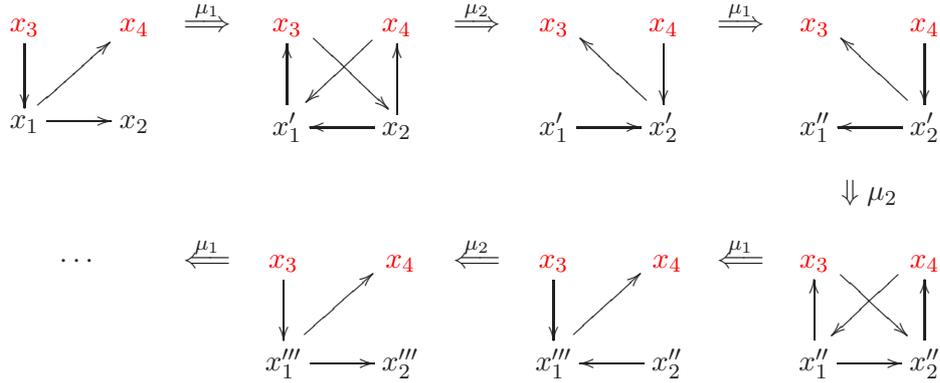
\begin{array}{ccccccc}
 \xymatrix{
{\color{red} x_3}\ar@{->}[d]&
{\color{red} x_4}\ar@<-1pt>@{<-}[ld]\\
x_1&x_2\ar@{<-}[l]
}
&
\stackrel{\mu_1}{\Longrightarrow}
&
 \xymatrix{
{\color{red} x_3}\ar@<3pt>@{->}[rd]\ar@{<-}[d]&
{\color{red} x_4}\ar@<-2pt>@{->}[ld]\ar@{<-}[d]\\
x_1'&x_2\ar@{->}[l]
}
&
\stackrel{\mu_2}{\Longrightarrow}
&
 \xymatrix{
{\color{red} x_3}\ar@<3pt>@{<-}[rd]&
{\color{red} x_4}\ar@{->}[d]\\
x_1'&x_2'\ar@{<-}[l]
}
&
\stackrel{\mu_1}{\Longrightarrow}
&
 \xymatrix{
{\color{red} x_3}\ar@<3pt>@{<-}[rd]&
{\color{red} x_4}\ar@{->}[d]\\
x_1''&x_2'\ar@{->}[l]
}
\\[50pt]
&&&&&&\Downarrow\mu_2\\[10pt]
\cdots
&\stackrel{\mu_1}{\Longleftarrow}&
\xymatrix{
{\color{red} x_3}\ar@{->}[d]&
{\color{red} x_4}\ar@<-3pt>@{<-}[ld]\\
x_1'''&x_2'''\ar@{<-}[l]
}
&\stackrel{\mu_2}{\Longleftarrow}&
\xymatrix{
{\color{red} x_3}\ar@{->}[d]&
{\color{red} x_4}\ar@<-3pt>@{<-}[ld]\\
x_1'''&x_2''\ar@{->}[l]
}
&
\stackrel{\mu_1}{\Longleftarrow}
&
 \xymatrix{
{\color{red} x_3}\ar@<3pt>@{->}[rd]\ar@{<-}[d]&
{\color{red} x_4}\ar@<-3pt>@{->}[ld]\ar@{<-}[d]\\
x_1''&x_2''\ar@{<-}[l]
}
\end{array}
$$
  Set
 $M : \mathbb{Z}^{2|2} \rightarrow \mathcal{F}$, which is determined by $\Lambda_M = \left(
                                                                                    \begin{array}{cccc}
                                                                                      0 & \lambda_1 & 0 & 0 \\
                                                                                      -\lambda_1 & 0 & 0 & 0 \\
                                                                                      0 & 0 & 0 & \lambda_2 \\
                                                                                      0 & 0&  - \lambda_2 & 0 \\
                                                                                    \end{array}
                                                                                  \right).
 $
Then $(\Lambda_M, \widetilde{\mathcal{Q}} )$ is compatible, and $(M, \widetilde{\mathcal{Q}})$ is the initial quantum super-seed. Let$X_i = M(e_i),$ then we have

\begin{eqnarray*}
 X_1' & = & M(- e_1) + M(- e_1 + e_2)+ (-1)^{\tau (e_{3},e_{4})}M(- e_1  + e_{3} + e_{4})\\
 & = & X_1^{-1} + q^{\frac{\lambda_1}{2}}X_1^{- 1}X_2 + q^{ - \frac{\lambda_2}{2}}X_1^{ - 1}X_3X_4 \\
 & = & X_1^{-1}(1 + q^{\frac{\lambda_1}{2}}X_2 + q^{ - \frac{\lambda_2}{2}}X_3X_4).
\end{eqnarray*}

We have $\Lambda_{M'} =\left(
                                       \begin{array}{cccc}
                                         0 & -\lambda_1 &  0 & 0\\
                                         \lambda_1 & 0 &0 & 0 \\
                                         0 & 0 & 0 & \lambda_2 \\
                                         0 & 0 &  - \lambda_2 & 0 \\
                                           \end{array}
                                           \right)
$
and $M'(e_1) = X_1^{-1} + q^{\frac{\lambda_1}{2}}X_1^{- 1}X_2 + q^{ - \frac{\lambda_2}{2}}X_1^{ - 1}X_3X_4,$ $ M'(e_2) = X_2$, $M(e_3) = X_3$, $M'(e_4) = X_4$. Therefore,

\begin{eqnarray*}
 X_2' & = & M'(- e_2) + M'(- e_2 + e_1)+ (-1)^{\tau (e_{3},e_{4})}M'(- e_2  + e_{3} + e_{4})\\
 & = & X_2^{-1} + q^{ - \frac{\lambda_1}{2}}X_1'X_2^{- 1} + q^{ - \frac{\lambda_2}{2}}X_2^{ - 1}X_3X_4 \\
 & = & X_1^{-1} + X_2^{-1} + q^{ - \frac{\lambda_1}{2}}X_1^{- 1}X_2^{- 1} + q^{ - \frac{\lambda_1 + \lambda_2}{2}}X_1^{- 1}X_2^{ - 1}X_3X_4  +q^{ - \frac{\lambda_2}{2}}X_2^{ - 1}X_3X_4 .
\end{eqnarray*}

Similarly, after computation, one obtain:

\begin{eqnarray*}
 X_1'' & = & M''(- e_1) + M''(- e_1 + e_2)\\
 & = & X_1'^{-1} + q^{  \frac{\lambda_1}{2}}X_1'^{-1}X_2'  \\
 & = &(1 + q^{\frac{\lambda_1}{2}}X_2 + q^{- \frac{\lambda_2}{2}}X_3X_4)^{-1}X_1[1 + q^{\frac{\lambda_1}{2}}X_2']\\
 & = &(1 + q^{\frac{\lambda_1}{2}}X_2 + q^{- \frac{\lambda_2}{2}}X_3X_4)^{-1}X_1[X_1X_2 + q^{\frac{\lambda_1}{2}}X_1 \\
 &&+ (1 + q^{\frac{\lambda_1}{2}}X_2 + q^{- \frac{\lambda_2}{2}}X_3X_4) + q^{\frac{\lambda_1 - \lambda_2}{2}}X_1X_3X_4]X_2^{-1}\\
 & = &(1 + q^{\frac{\lambda_1}{2}}X_1)X_2^{-1}.
\end{eqnarray*}

Similarly, there are:

$$X_2'' = X_1( 1 - q^{- \frac{\lambda_2}{2}}X_3X_4),\ X_1''' = X_2(1 - q^{- \frac{\lambda_2}{2}}X_3X_4),\ X_2''' = X_1^{-1} + q^{\frac{\lambda_1}{2}}X_1^{- 1}X_2 + q^{ - \frac{\lambda_2}{2}}X_1^{ - 1}X_3X_4.$$

As we see, the mutation is not involution, and the process is still infinite and aperiodic. The corresponding quantum cluster superalgebra is $\mathbb{Z}[q^{\pm \frac{1}{2}}]$ subalgebra of $\mathcal{F}$ generated by $X_1, X_2, X_3, X_4, X_1',X_2',X_1''$.
\end{example}

\section{Quantum Laurent Phenomenon}

Let $(M, \widetilde{\mathcal{Q}})$ be quantum super-seed in $\mathcal{F}$. $\widetilde{\mathbf{X}} = \{X_i | X_i = M(e_i)\}$, let $\mathcal{T}(\widetilde{\mathbf{X}})$ denote quantum super torus generated by $\widetilde{\mathbf{X}}$. In this section, we assumed that $\widetilde{\mathbf{X}}$ is numbered so that its cluster $\mathbf{X}$ has the form $(X_1,X_2,\cdots,X_l)$. Thus,  we have $\mathbf{C} = \widetilde{\mathbf{X}} - \mathbf{X}$. The ground ring $\mathbb{ZP}$ is $\mathbb{Z}[q^{\pm\frac{1}{2}}, X_{l+1}^{\pm 1},\cdots,X_n^{\pm 1},X_{n+1},\cdots,X_{n + m}]$. $\mathbf{X}_k$ denote the cluster of $(M_k, \widetilde{\mathcal{Q}}_k) = \mu_k((M, \widetilde{\mathcal{Q}})).$ Then
$$\mathbf{X}_k = \mathbf{X} - \{X_k\} \cup \{X_k'\}$$

We denote $\mathcal{U}(M,\widetilde{\mathcal{Q}}) \subset \mathcal{F}$  the $\mathbb{ZP}$-subalgebra of $\mathcal{F}$ given by

$$\mathcal{U}(M,\widetilde{\mathcal{Q}}) = \mathbb{ZP}[\mathbf{X}^{\pm 1}] \cap \mathbb{ZP}[\mathbf{X}_1^{\pm 1}] \cap \cdots \cap \mathbb{ZP}[\mathbf{X}_l^{pm 1}]\subset \mathcal{F} .$$

\begin{lem}
(1) Every element $Y \in \mathbb{ZP}[\mathbf{X}]$ can be uniquely written in the form

$$Y = \sum\limits_{r \in \mathbb{Z}}c_r X_1^r,$$
where each coefficient $c_r$ belong to $\mathbb{ZP}[X_2^{\pm 1}, \cdots, X_l^{\pm 1}]$, and all but finitely many of them are equal to 0.

(2) Every element $Y \in  \mathbb{ZP}[\mathbf{X}^{\pm 1}] \cap \mathbb{ZP}[\mathbf{X}_1^{\pm 1}]$ can be uniquely written in the form

$$Y = c_0 + \sum_{r \geq 0}(c_r X_1^r + c_r'(X_1')^r)$$
where all coefficient $c_r$ and $c_r'$ belong to $\mathbb{ZP}[X_2^{\pm 1}, \cdots, X_l^{\pm 1}]$, and all but finitely many of them are equal to 0.

(3)An element $Y \in \mathcal{F}$ belongs to $\mathbb{ZP}[X_1,X_1',X_2^{\pm 1}, \cdots , X_1^{\pm 1}]$ if and only if it has the form $Y = \sum\limits_{r \in \mathbb{Z}}c_r X_1^r$, and for $r > 0$, the coefficient $c_{- r}$ is divisible by $P_{b^1}^r$ in algebra $\mathbb{ZP}[X_2^{\pm 1}, \cdots, X_l^{\pm 1}]$, where $P_{b^1}^r$ is in the following proof.
\end{lem}

\begin{proof}
 Consider the ring $\mathbb{ZP}[\mathbf{X}^{\pm 1}]$. Since $X_i$ and $X_j$ are quasi-commuting, thus the first part of lemma follows.

 We can view the $j$ column $b^j$ as the element of $\mathbb{Z}^{n | m}$. Define the elements in $\mathcal{T}$ as follows

 $$P_{b^j}^r=\prod\limits_{p = 1}^{r}(1 + q^{\frac{(1-2p)d(b^j)}{2}}X^{-b^j} + \sum\limits_{{(k,k')\in S}}(-1)^{\tau(e_k,e_k')X^{e_k + e_{k'}}})$$
where $d(b^j)$ denote the minimal positive integer of $\Lambda_{b^j, e}$ for $e \in \mathbb{Z}^{n | m}$. $S = \{(k,k')| k,k' > n, \tilde{b}_{k,j} = 1,\tilde{b}_{k'j} = -1\}$.
 $P_{b^j}^r$ is the center of $\mathbb{ZP}[X_1^{\pm 1}, \cdots, X_{j - 1}^{\pm 1}, X_{j + 1}^{\pm 1}, \cdots, X_l^{\pm 1}].$

 Though directly compute, we have

 $$(X_1')^r = P_{b^1}^rX^{e_1'},$$

 where

 $$e_1' = - e_1 + \sum\limits_{b_{i1} > 0}b_{i1}e_i.$$

To prove the second part of the lemma, note that any $Y \in \mathbb{ZP}[\mathbf{X}^{\pm 1}]$ is of the form

$$Y = \sum\limits_{r = - N}^{N} c_rX_1^r,$$

there exist $P_1 \in \mathbb{ZP}[X_2,X_3, \cdots, X_l]$ such that $X_1 = P_1(X_1')^{-1}$.Then

\begin{eqnarray*}
 Y & = &  \sum\limits_{r = 0}^{N} c_rX_1^r + \sum\limits_{r = 1}^{N} c_{- r}X_1^{- r}\\
 & = &   \sum\limits_{r = 0}^{N} c_r(P_1(X_1')^{-1})^r + \sum\limits_{r = 1}^{N} c_{- r}'(P_{b^1}^r)^{-1}(P_{b^1}^r)(X^{e_1'})^{r}.
\end{eqnarray*}

In addition, $Y \in \mathbb{ZP}[\mathbf{X}_1^{\pm 1}]$, then $c_{- r}$ is divisible by $P_{b^j}^r$ in algebra $\mathbb{ZP}[X_2^{\pm 1}, \cdots, X_l^{\pm 1}]$.

\end{proof}

\begin{prop}
Suppose $l \geq 2$, then

$$\mathcal{U}(M, \widetilde{\mathcal{Q}}) = \bigcap\limits_{j = 2}^n \mathbb{ZP}[X_1,X_1', X_2^{\pm 1},\cdots, X_{j - 1}^{\pm 1}, X_j,X_j',X_{j +1}^{\pm 1}, \cdots, X_l^{\pm 1}].$$
\end{prop}
\begin{proof}
 The proof is similarly to \cite{BZ1}. We leave the details to the reader.
\end{proof}

\begin{lem}
\label{lem:rank2}
In the above notation, suppose $l = 2$. we have
$$\mathbb{ZP} [X_1,X'_1, X_2,X'_2]
= \mathbb{ZP}[X_1,X'_1, X_2,X''_2].$$
\end{lem}

\begin{proof}
By symmetry, it is enough to show that
\begin{equation}
\label{eq:X2''}
X''_2 \in \mathbb{ZP}[X_1,X'_1, X_2,X'_2]\ .
\end{equation}
Without loss of generality, we consider the following extended quiver:

$$
\xymatrix@R=0.5cm{
{\color{red}I_i=\{\xi_1\cdots\xi_r\}}\ar@{->}[rd]&
{\color{red}J_i=\{\eta_1\cdots\eta_s\}}\ar@{<-}[d]&
{\color{red}I_j=\{\zeta_1\cdots\zeta_t\}}\ar@{->}[d]&
{\color{red}J_j=\{\e_1\cdots\e_u\}}\ar@{<-}[ld]\\
&x_i\ar@{->}[r]^{r}& x_j&
}
$$
where $r>0$ denotes the numbers of arrows from $x_i$ to $x_j$.
By the definition of quantum super-seed, $d:=r \lambda_{12}>0.$
\[
\begin{array}{ccl}
  X_1' & = & M(-e_1)+M(-e_1+re_2)+ \sum\limits_{i\in I_i, j\in J_i}(-1)^{\tau({e_i,e_j})}  M(-e_1+e_i+e_j) \\
   & =& (1+q^{-\frac{d}{2}}X_2^{r}+\sum\limits_{i\in I_i, j\in J_i}(-1)^{\tau({e_i,e_j})}  M(e_i+e_j))X^{-e_1}.
\end{array}
\]
A direct check shows that
\begin{equation}\label{x1r}
 (X_1')^r=\prod\limits_{s=1}^{r}(1+\sum\limits_{i\in I_i, j\in J_i}(-1)^{\tau({e_i,e_j})}  M(e_i+e_j)+q^{-\frac{d}{2}(2s-1)}X^{re_2})X^{-re_1} .
\end{equation}
Denote $P:=\prod\limits_{s=1}^{r}(1+\sum\limits_{i\in I_i, j\in J_i}(-1)^{\tau({e_i,e_j})}  M(e_i+e_j)+q^{-\frac{d}{2}(2s-1)}X^{re_2})
=\prod\limits_{s=1}^{r}(C+q^{-\frac{d}{2}(2s-1)}X^{re_2})$.

\[
\begin{array}{ccl}
  X_2' & = & M(-e_2)+M(-e_2+re_1)+ \sum\limits_{k\in I_j, l\in J_j}(-1)^{\tau({e_k,e_l})}  M(-e_2+re_1+e_k+e_l) \\
   & =& X_2^{-1}+q^{-\frac{d}{2}}X_2^{-1}X^{re_1}+\sum\limits_{k\in I_j, l\in J_j}(-1)^{\tau({e_k,e_l})}q^{-\frac{d}{2}}X_2^{-1}X^{re_1}M(e_k+e_l).
\end{array}
\]

\[
 X_2 X_2' = 1+q^{-\frac{d}{2}}X^{re_1}+\sum\limits_{k\in I_j, l\in J_j}(-1)^{\tau({e_k,e_l})}q^{-\frac{d}{2}}X^{re_1}M(e_k+e_l).
\]
Then, we have
\begin{equation}\label{1}
  1=X_2 X_2'-q^{-\frac{d}{2}}X^{re_1}-\sum\limits_{k\in I_j, l\in J_j}(-1)^{\tau({e_k,e_l})}q^{-\frac{d}{2}}X^{re_1}M(e_k+e_l).
\end{equation}

On the other hand, $X_2^{''}=X_2^{-1}+q^{\frac{d}{2}}X_2^{-1}(X_1')^r+\sum\limits_{u\in I_j', v\in J_j'}(-1)^{\tau({e_u,e_v})}M(-e_2+e_u+e_v),$
where $I_j'=I_j\cup I_i$, $J_j'=J_j\cup J_i$.

Using \eqref{1},we rewrite
\[
\begin{array}{ccl}
 X_2^{''} & =& X_2^{-1}+q^{\frac{d}{2}}X_2^{-1}(X_1')^r+\sum\limits_{u\in I_j', v\in J_j'}(-1)^{\tau({e_u,e_v})}M(-e_2+e_u+e_v) \\
 & = & X_2^{-1}+\sum\limits_{u\in I_j', v\in J_j'}(-1)^{\tau({e_u,e_v})}M(-e_2+e_u+e_v)\\
 &  &+ q^{\frac{d}{2}}X_2^{-1}(X_1')^r(X_2 X_2'-q^{-\frac{d}{2}}X^{re_1}-\sum\limits_{k\in I_j, l\in J_j}(-1)^{\tau({e_k,e_l})}q^{-\frac{d}{2}}X^{re_1}M(e_k+e_l))\\
 & = & S_1-S_2+S_3,
\end{array}
\]
where
\[
\begin{array}{l}
\medskip
  S_1=q^{\frac{d}{2}}X_2^{-1}(X_1')^rX_2 X_2', \\
\medskip
  S_2=q^{\frac{d}{2}}X_2^{-1}(P-C)X^{-re_1}(q^{-\frac{d}{2}}X^{re_1}-\sum\limits_{k\in I_j, l\in J_j}(-1)^{\tau({e_k,e_l})}q^{-\frac{d}{2}}X^{re_1}M(e_k+e_l)), \\
\medskip
  S_3=X_2^{-1}+\sum\limits_{u\in I_j', v\in J_j'}(-1)^{\tau({e_u,e_v})}M(-e_2+e_u+e_v)\\
\medskip
  -q^{\frac{d}{2}}X_2^{-1}CX^{-re_1}(q^{-\frac{d}{2}}X^{re_1}-\sum\limits_{k\in I_j, l\in J_j}(-1)^{\tau({e_k,e_l})}q^{-\frac{d}{2}}X^{re_1}M(e_k+e_l))
\end{array}
\]
To complete the proof, we will show that
$$S_1, S_2 \in \mathbb{ZP}[X_1,X'_1, X_2,X'_2],
\quad S_3 = 0.$$

First, using \eqref{x1r},
$$ S_1=q^{\frac{d}{2}}X_2^{-1}PX^{-re_1}X_2 X_2'=q^{-\frac{d}{2}}PX^{-re_1}X_2'=q^{-\frac{d}{2}}(X_1')^r X_2'. $$

For $S_2$, we notice that $P-C$ is a polynomial in $X_2$ with coefficients in $\mathbb{Z}[q^{\pm 1/2}]$ and zero constant term.

Finally, $S_3=0$ is guaranteed by the assumption that the mutation $\mu_i$ is allowed.
This complete the proof of Lemma ~\ref{lem:rank2}.
\end{proof}

Since the quantum cluster superalgebra $\mathcal{A}(\widetilde{Q})$ is a subalgebra of $\mathcal{U}(M, \tilde B)$,together with the above several lemmas. We obtained the main theorem of this section.

\begin{thm}
For every extended quiver~$\widetilde{\mathcal{Q}}$,and the initial quantum super-seed $(M,\widetilde{\mathcal{Q}})$, the cluster variable $Y$
obtained by any series of allowed mutations are in the ring $\mathbb{ZP}[\mathbf{X}^{\pm 1}]$.
\end{thm}

\noindent{\bf Acknowledgments}\, \,
This work is supported by the NSFC under numbers 11475178 and 11571119.

\end{document}